\documentclass{amsart}
\usepackage{graphicx}
\usepackage[english]{babel}
\usepackage{euscript}
\usepackage{amsfonts}
\usepackage{amsmath, amscd, amssymb}
\usepackage{latexsym}
\def \SO{\mathop\mathrm{SO}}
\def \SL{\mathop{SL}}
\def \R{\mathbb{R}}
\def \Z{\mathbb{Z}}
\newcommand{\LL}{{\mathcal L}}
\def \e{\varepsilon}
\def \F{F}

\def \M{\mathcal M}

\newtheorem{theorem}{Theorem}

\newtheorem{lemma}[theorem]{Lemma}
\newtheorem{corollary}[theorem]{Corollary}
\newtheorem{proposition}[theorem]{Proposition}

\theoremstyle{remark}
\newtheorem*{remark}{Remark}

\begin{document}
\title[Space of non-degenerate curves in a Riemannin manifold]{The space of non-degenerate closed curves in a Riemannian manifold}
\author{J.~Mostovoy}
\address{CINVESTAV, Col. San Pedro Zacatenco,
M\'exico, D.F. CP 07360}
\email{jacob.mostovoy@gmail.com}
\author{R.~Sadykov}
\address{CINVESTAV, Col. San Pedro Zacatenco,
M\'exico, D.F. CP 07360}
\email{rstsdk@gmail.com}


\date{}

\begin{abstract}

Let  $\LL M$ be the semigroup of non-degenerate based loops with a fixed initial/final frame in  a Riemannian manifold $M$ of dimension at least three. 
We compare the topology of $\LL M$ to that of the loop space $\Omega FTM$ on the bundle of frames in the tangent bundle of $M$. We show that $\Omega FTM$ is the group completion of $\LL M$, and prove that it is obtained by localizing $\LL M$ with respect to adding a ``small twist''.
\end{abstract}

\maketitle

\section{Introduction}
A smooth curve in an $n$-dimensional Riemannian manifold $M$ is called {\em free of order $k$} if its first $k$ covariant derivatives are linearly independent at each point. Free curves of order smaller than $n$ were studied by Smale \cite{Smale}, Feldman \cite{Feldman} and Gromov \cite{Gromov}, who showed that they satisfy the $h$-principle. In particular,  when $k<n$, the space of all based free curves of order $k$ is homotopy equivalent to the loop space on the bundle of $k$-dimensional frames in the tangent bundle of $M$.

When $k=n$, the $h$-principle fails to be true already for the simplest examples $M=S^2$, see \cite{Li70}, and $M=\R^{2m}$ with $m>1$, see \cite{Sh}. We show that a somewhat weaker statement holds: for any $n>2$ the loop space on the bundle of $n$-frames in $TM$ is the group completion of the monoid of free curves of order $n$ in $M$. Moreover, the group completion is achieved by localizing with respect to the multiplication by any ``small'' curve.

The spaces of free curves of order $n$ in $n$-manifolds, also called {\em non-degenerate} curves (the terminology varies in the literature) are the subject of an extensive and ongoing study, with significant advances in special cases \cite{NewSal}. In particular, our result builds heavily on the work of M.~Shapiro \cite{Sh} and N.~Saldanha and B.~Shapiro \cite{ShSal}. The main technical tool that we borrow from  \cite{ShSal} are the ``telephone wires'' which allow to approximate in an appropriate sense any curve  by a non-degenerate one. This idea goes back to Gromov (convex integration, \cite{Gromov}), Eliashberg-Mishachev
(holonomic approximation theorem, \cite{EM}) and Thurston (corrugations, \cite{OutsideIn}),
and in some form was used by Rourke and Sanderson (compression theorem, \cite{RS}).

A strong motivation for much of the research in the area is the fact that non-degenerate curves appear in the study of linear ordinary differential equations~\cite{BST06,BST09,KhSh,Sh90,ST}. For example, the space of linear ordinary differential equations of order $n+1$ on $S^1$ that have $n+1$ independent solutions is homotopy
equivalent to the space of based non-degenerate curves in $S^n$.

\section{Statement of the results}
A $C^n$-differentiable curve $\gamma$ in a Riemannian manifold $M$ of dimension $n$ is said to be {\em non-degenerate} if at each moment of time $t$ the set of covariant derivatives
$$\{ \gamma'(t),\ldots, \gamma^{(n-1)}(t), \gamma^{(n)}(t)\}$$
spans the tangent space $T_{\gamma(t)} M$. Here $\gamma^{(k)}(t)$ for $k>1$ is defined as $\nabla \gamma^{(k-1)}(t)$. If the curve $\gamma$ is non-degenerate, the ordered set of its first $n$ covariant derivatives is referred to as the {\em frame of $\gamma$} at time $t$ and is denoted by  $\F_\gamma (t)$. 

Assume that a basis $F_0$ is chosen in the tangent space to $M$ at a point $x_0\in M$.  Consider the topological space
$\LL M$ consisting of all non-degenerate curves
$$\gamma: [0,a] \to M,$$
such that
$$\gamma(0)=\gamma(a)=x_0$$
and
$$\F_{\gamma}(0)=\F_{\gamma}(a)=F_0,$$
where $a$ is a positive number.

The space $\LL M$ is a topological semigroup: the multiplication is given by concatenating the curves.  Let $FTM$ be the bundle of frames in the tangent bundle $TM$, with the basepoint being the frame $F_0$, and $\Omega FTM$ be its Moore loop space. The {\em frame map}
$$\LL M \to \Omega FTM$$
is defined as
$$\gamma \to F_\gamma.$$
It is a homomorphism of topological semigroups.

\begin{remark} Often, the frame of the derivatives of a curve is orthogonalized and normalized; the result is called the {\em Frenet frame}. Accordingly, instead of the frame map that we consider here it is  customary to consider the Frenet frame map. If the space of all frames is identified with the Lie group $\SL(n)$, the orthonormal frames correspond to $\SO(n)$. The two groups are homotopy equivalent and, in the context of the present paper, it will make no difference which version is used to state the results.
\end{remark}

Observe that if we add a disjoint basepoint to $\LL M$, the resulting space $\LL M_+$ is a topological monoid: the basepoint can be thought of as representing the ``curve of zero length'' and is the neutral element with respect to the concatenation.  We shall prove the following result:
\begin{theorem}\label{completion}
For a Riemannian manifold $M$ of dimension at least three, the homotopy-theoretic group completion $\Omega B \LL M_+$ of the monoid $\LL M_+$ is homotopy equivalent to $\Omega FTM$.
\end{theorem}

In fact, we shall describe the group completion of $\LL M_+$ explicitly, in terms of the localization with respect to the left multiplication by one fixed element. 

\medskip
 
In the case when $M=\R^n$ we take the origin to be the basepoint and the standard basis to be the chosen frame. The space $FT\R^n$ is homeomorphic to $ \SL(n)\times \R^n$.
For a general $M$, identifying $\R^n$ with $T_{x_0} M$ in the manner that preserves the chosen frames, we get the exponential map
$$\exp_{x_0}: B_{R}^n\to M$$
defined on an open ball $B_{R}\subset \R^n$ of radius $R$; we can take $R$ to be infinite if $M$ is complete. Take an arbitrary  curve $\alpha$ in $\LL\R^n$.
We shall see that there exists a positive number $\lambda_0$ such that for all positive $\lambda\leq\lambda_0$ the curves $\lambda\alpha$ are in $B_{R}^n$ and their images $\exp_{x_0}(\lambda\alpha)$ in $M$ are all non-degenerate. Define $\omega\in\LL M$ as
$$\omega=\exp_{x_0}(\lambda_0\alpha).$$

Write $\LL M_\omega$ for the localization of $\LL M$ by the left multiplication by $\omega$, that is,  the direct limit of the sequence of embeddings
$$\LL M\stackrel{\omega\cdot}{\longrightarrow}\LL M\stackrel{\omega\cdot}{\longrightarrow}\LL M\stackrel{\omega\cdot}{\longrightarrow}\ldots$$
The frame map sends the left multiplication by $\omega$ in $\LL M$ to the left multiplication by $F_{\omega}$ in $\Omega FTM$. This latter map is a homotopy equivalence, and the direct limit of a sequence of multiplications by $F_\omega$ in $\Omega FTM$ is again homotopy equivalent to $\Omega FTM$. In particular, the frame map descends to a map $$\LL M_\omega\to \Omega FTM.$$

The main technical result of this note is the following statement:
\begin{theorem}\label{main}
For a Riemannian manifold $M$ of dimension at least three, the space  $\LL M_\omega$  is  weakly homotopy equivalent to $\Omega FTM$ and the equivalence is given by the  frame map.
\end{theorem}

Theorem~\ref{main} is proved in Section~\ref{next} and in Section~\ref{theoneafter} we show how it implies Theorem~\ref{completion}.


\section{Proof of Theorem~\ref{main}}\label{next}

\subsection{The proof for non-degenerate curves in $\R^n$}
It is known \cite{Li71, Sh, ShaSha} that the set $\pi_0 \LL\R^n$ consists of two elements when $n$ is odd and of three elements when $n$ is even. The concatenation of curves gives  $\pi_0 \LL\R^n$ the structure of a semigroup and the  frame map
$$\pi_0 \LL\R^n\to\pi_0 \Omega \SL(n)=\Z/2$$
is an isomorphism when $n$ is odd and a group completion when $n$ is even. Of the three components of $\pi_0 \LL\R^{2k}$ two are mapped to the generator of $\Z/2$; of these two one is sent to the other by the left multiplication by  the class of 
any element in $\LL\R^{2k}$ defining a contractible loop in $\SL(2k)$ (see \cite{Sh}).  In particular, the localization of $\pi_0 \LL\R^n$ with respect to the left multiplication by any of its elements is the same thing as its group completion, and is isomorphic to $\Z/2$.

Therefore, we only need to show that each component of $\LL\R^n_\omega$ has the same homotopy groups as a component of $\Omega \SL(n)$ and the isomorphism is induced by the  frame map. This, essentially, was proved by Saldanha and Shapiro \cite{ShSal}. Their argument goes as follows.

Fix $\omega\in\LL \R^n$ parametrized by $[0,1]$ and consider a curve $$A:[0,a]\to \SL(n)$$
with the property that $A(t)=\rm{Id}$ when $t\in[0,\e]\cup [a-\e,a]$
for some small $\e>0$. For a positive integer $N$ let
$$A^{[N]}:[0,a]\to \R^n$$ be the curve defined as $$A^{[N]}(t)=A(t)\omega(Nt/a).$$
In what follows, when referring to a {\em compact family of curves} $\gamma_r$ we shall mean that $r$ varies over a compact set.
\begin{lemma}\label{sasha}
Given a compact family of curves $A_r\in \Omega SL(m)$, each $A_r$ parametrized by $[0,a_r]$ and constant in a neighbourhood of $\{0\}\cup\{a_r\}$,
we can find $N$ such that
\begin{itemize}
\item each curve  $A_r^{[N]}$ is non-degenerate;
\item the  frame map sends the family $A_r^{[N]}$ to a family of curves homotopic to $A_r$.
\end{itemize}
Moreover, if the curve $A_r$ is in the image of the  frame map, that is, $A_r=F_{\gamma_r}$, then $A_r^{[N]}$ is homotopic to $\gamma_r\cdot \omega^N$ through non-degenerate curves. 
\end{lemma}

\begin{proof}
Indeed, the non-degeneracy of $A_r^{[N]}$ follows from the fact that
asymptotically, as $N\to \infty$, we have
$$\bigl(A(t)\omega(Nt/a)\bigr)^{(k)}\sim N^k\cdot A(t)\omega^{(k)}(Nt/a).$$
To prove the second claim, take $N$ to be even. Cut the curve $A$ into  $N/2$ pieces 
$$A_i=A|_{[2(i-1)a/N, 2ia/N]}$$ and write $f_i$ for the curve $A(2ia/N)
F_{\omega^2}$. Consider the concatenation of curves
$$A_1f_1 A_2 f_2\ldots A_{N/2} f_{N/2}.$$
On one hand, this curve is homotopic to $A$. This is obvious since $F_{\omega^2}$ is null-homotopic. On the other
hand, it is close, and homotopic,  to $F_{A^{[N]}(t)}$ if $N$ is sufficiently big.
See \cite{ShSal} for details.

If $A$ is in the image of the  frame map, $A=\F_{\gamma}$, the above cut and paste construction can be described in terms of non-degenerate curves. One cuts the curve $\gamma$ into $N/2$ segments of equal length and inserts a suitably translated copy of $\omega^2$ after each segment; then one should take the  frame map. Now, if we shrink the lengths of all the segments of $\gamma$, apart from the first one, to zero, while preserving the total length, we obtain the curve $\gamma\cdot \omega^N$. Performing this shrinking uniformly on all the curves of a family $\gamma_r$, we get a homotopy between the family $\gamma_r\cdot\omega^N$ and a family which is close and, hence, homotopic, to $A_r^{[N]}$. 

\end{proof}

As a consequence of Lemma~\ref{sasha}, all the elements of all the homotopy groups of $\Omega \SL(n)$ can be represented by families of
 frames of non-degenerate curves. Moreover, 
if two classes in $\pi_k\LL\R^n$ have the same image in $\pi_k \Omega \SL(n)$, they become equal in $\pi_k \LL\R^n_\omega$.

\subsection{Piecewise $C^n$ curves in Riemannian manifolds}

In order to extend the argument to arbitrary Riemannian manifolds, we need to enlarge the space of non-degenerate curves by the curves whose frame can experience a small jump in a finite number of points. 

Let us say that two $n$-frames $L_1$ and $L_2$ in $\R^n$ are $\e$-close if $L_1L_2^{-1}$ and $L_2 L_1^{-1}$, thought of as elements of $\SL(n)$, are $\e$-close to the identity in  the usual matrix norm. Define $\LL M (\delta)$ as the space of $C^1$-curves which are piecewise $C^n$ and non-degenerate, with the initial and the final frame $\delta$-close to the chosen frame in $M$ at $x_0$, and such that the limits of the frames on the right and on the left at any point are $\delta$-close. The topology on $\LL M (\delta)$ is given by the Sobolev $H^n$ metric.  

\begin{proposition}\label{smoo}
The inclusion map $\LL M\to \LL M(\delta)$ is a weak homotopy equivalence for sufficiently small $\delta$.
\end{proposition}
\begin{proof}
We have to show that any compact family of curves $\gamma_r$ in $\LL M(\delta)$ can be deformed into $\LL M$ by a homotopy that sends $\LL M$ to itself. 
Choose $\e_1,\e_2>0$, a curve $\gamma_0\in \LL M$ parametrized by $[0,1]$, and a family of functions $f_\e:[0,1]\to [0,1]$ as in the figure such that all the derivatives of $f_\e$ vanish at $0$ and $1$: 
$$\includegraphics{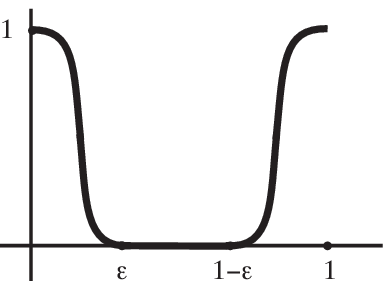}$$ 
Write $a_r$ for the length of the interval that parametrizes $\gamma_r$.

First, we perform a smoothing of each $\gamma_r$ with the help of the standard mollifier with the parameter $\tau\in [0,\e_1]$ (see \cite{Evans}). If $\e_1$ is sufficiently small, the resulting curves ${\gamma}_{r, \tau}$ will be non-degenerate, but their initial and final frames might differ from the chosen frame $F_0$. In order to remedy this, we replace  ${\gamma}_{r, \tau}$ with $$\tilde{\gamma}_{r, \tau}(t) = {\gamma}_{r, \tau}(t) (1-f_{\e_2}(a_r t))+\gamma_0(a_r t) f_{\e_2}(a_r t).$$
One can choose $\e_1$ and $\e_2$ so small that $\tilde{\gamma}_{r, \tau}$ is non-degenerate for all $\tau\in [0,\e_1]$ and that  $\tilde{\gamma}_{r, 0}$ is homotopic to $\gamma_r$. This gives the desired deformation.
\end{proof}

\subsection{The proof for curves in arbitrary $M$}
First, we describe the behaviour of non-degenerate curves under the exponential map.
\begin{lemma}\label{lemma:exp}
Let $\alpha:[0,a]\to T_{x}M$ be a non-degenerate curve in the tangent space to $x\in M$. There exists a positive $\lambda_0$ such that for all positive $\lambda\leq\lambda_0$ the curve $\exp_{x}(\lambda\alpha)$ in $M$ is non-degenerate.
\end{lemma}

\begin{proof}
In the normal coordinates around $x_0$ the exponential map is,
tautologically, the identity map. Therefore, all we need to prove is
that the difference between the usual derivatives, suitably scaled,
and the covariant derivatives of $\lambda\alpha$ along the vector
field $\lambda\alpha'$ is of the higher order in $\lambda$ than the
derivatives themselves.

This is clear for the first derivatives, since they simply coincide.
For the second derivatives, we have
$$(\lambda\alpha)^{(2)}=\nabla_{\lambda \alpha'} \lambda\alpha' =  \lambda^2(\alpha'' + \Gamma_{ij}^k\alpha'_i\alpha'_j\cdot e_k),$$
where $e_k$ are the basis vectors and $\alpha'_i$ is the $i$th component of $\alpha'$. Since for the metric connection
the Christoffel symbols $\Gamma_{ij}^k(\lambda\alpha)$ tend to zero
with $\lambda$, the case $n=2$ is also settled.

In general, the differences between the derivatives will involve
Christoffel symbols $\Gamma_{ij}^k(\lambda\alpha)$ and their
derivatives, which tend to zero even faster with $\lambda$. Hence,
for $\lambda$ small enough the first $n$ covariant derivatives of
$\lambda\alpha$ will be linearly independent for each $t\in [0,a]$.
\end{proof}

In fact, Lemma~\ref{lemma:exp} holds for compact families of curves.
Namely, for a family of curves $\alpha_r\subset T_{x_r}M$, where $r$
ranges over a compact set, we can find $\lambda_0$ as in the
statement of Lemma~\ref{lemma:exp}, the same for all the curves
$\alpha_r$. This follows directly from continuity of the Christoffel
symbols and their derivatives. As a consequence, we have:
\begin{corollary}\label{another}
For all $k\geq 0$ there is a well-defined map $\pi_k \LL \R^n\to  \pi_k \LL M$.
\end{corollary}
\begin{proof}
Indeed, since the Christoffel symbols vanish at the origin, for any given $\delta$ and for any compact family $\gamma_r\in\LL \R^n$ we can find $\lambda_0$ small enough so that for all $\lambda<\lambda_0$ the curves $\exp(\lambda\gamma_r(t/\lambda))$ are well-defined and lie in $\LL M (\delta)$. Applying Proposition~\ref{smoo} we get the result.
\end{proof}

Now, let $$\Gamma:[0,a]\to FTM$$ be an arbitrary curve and
$$\gamma:[0,a]\to M$$ its projection to $M$. The pullback
$\gamma^*TM$ of the tangent bundle of $M$ to $[0,a]$ has a
trivialization given by $\Gamma$, and we can identify it with
$[0,a]\times \R^n$. There exists a positive $\e$ such that the map
$$
\exp_\Gamma: [0,a]\times B^n_\e \to M$$ given by $$(t, v) \to
\exp_{\gamma(t)} (v)$$ is well-defined. Clearly,  the restriction of
$\exp_\Gamma$ to $[0,a]\times 0$ is the curve $\gamma$. We stress
that $\exp_\Gamma$ depends on the choice of a basis in each tangent
space $T_{\gamma(t)}M$ and not just on the curve $\gamma$.
\medskip

Let us now prove that $\gamma$ can be approximated by a
non-degenerate curve, and, moreover, that this can be done in
families.

\medskip

Assume that $\alpha: [0,1]\to\R^n$ lies in $\LL\R^n$. 
It will be convenient to extend $\alpha$ to a periodic function on $\R$, with period 1, and keep the same notation for it. Let $\gamma_r: [0,a_r]\to M$ be a compact family of curves and let
$\Gamma_r: [0,a_r]\to FTM$ be a lifting of $\gamma_r$. 
\begin{lemma}\label{appr}
There exist
$\lambda_0>0$ and $N_0\in \mathbb N$ such that for all $N>N_0$ each curve
$${\gamma}_r^{[N]}=\exp_{\Gamma_r}\left(\frac{t}{\lambda_0 N},  \lambda_0 \alpha\left(\frac{t}{\lambda_0 a_r}\right)\right),$$ defined on the interval $[0, a_r\cdot \lambda_0 N]$,  is non-degenerate  in $M$. Moreover, for any $\delta>0$ the number $N_0$ can be chosen so that  the initial and the final frames of ${\gamma}_r^{[N]}$ are  $\delta$-close to the corresponding frames of $\gamma_r$.
\end{lemma}
Essentially, ${\gamma}_r^{[N]}$ is a ``telephone wire'' in the shape of $\gamma_r$, as in \cite{ShSal}.
\begin{proof}
For each $r,t$ we have an identification of $T_{\gamma_r(t)} M$ with $\R^n$ and thus we can speak of a copy of $\alpha$ in $T_{\gamma_r(t)} M$. By Lemma~\ref{lemma:exp} (or, rather, its version for families) there exists $\lambda_0$ such that the curve $\exp_{\gamma_r(t)} (\lambda \alpha)$ is non-degenerate in $M$ for all $\lambda\leq \lambda_0$. 
By rescaling $\alpha$ and reparametrizing $\gamma_r$ we can assume, without loss of generality, that $\lambda_0=1$ and $a_r=1$ for all $r$. In particular, we have that
$${\gamma}_r^{[N]}=\exp_{\Gamma_r}(t/N,   \alpha(t))$$ 
is defined on the interval $[0,N]$.

Now, for each $\e>0$ we can find $N_0$ big enough so that each curve $\exp_{\Gamma_r} (t/N, 
\alpha(t))$, restricted to $t\in [i, i+1]$ has all its covariant derivatives $\e$-close to
those of the curve $\exp_{\Gamma_r} (i/N, \alpha(t))$, for all
$N>N_0$ and for $i=0,\ldots, N-1$. 
In particular, this means that  we can choose $N_0$ so that
$\exp_{\Gamma_r} (t/N,  \alpha(t))$ for all $N>N_0$ is
non-degenerate. Moreover, this also implies the second statement of the lemma.
\end{proof}

In particular, if the family $\gamma_r$ is in $\LL M$ to begin with and $\Gamma_r$ is given by the  frame map, we see that for any $\delta>0$ the number $N$ can be chosen so that ${\gamma}_r^{[N]}\in \LL M(\delta)$. In fact, the homotopy class of  ${\gamma}_r^{[N]}$ can be explicitly described. 

\begin{lemma}\label{appr2}
If $\gamma_r\in\LL M$ for all $r$, $\Gamma_r$ is given by the  frame map, and $N_0$ is sufficiently big, 
then for all $N>N_0$ the family ${\gamma}_r^{[N]}$ is homotopic in $\LL M (\delta)$ to the family $\gamma_r\cdot \omega^N$, where $\omega=\exp_{x_0}\alpha$.
\end{lemma}
\begin{proof}
Let us assume, as in the previous proof, that $a_r=1$ and $\lambda_0=1$. Given $N\in \mathbb N$ we write
$$\omega_i:=\exp_{\Gamma_r} (i/N, \alpha(t)),$$
defined on $[i,i+1]$, and
$$\gamma_r^i:=\gamma_r|_{[i/N,(i+1)/N]}.$$
 
Choose $\e>0$. As we have noted before, for $N$ big enough  ${\gamma}_r^{[N]}$, restricted to $t\in [i, i+1]$ has its first $n$ covariant derivatives $\e$-close to those of the curve $\omega_i$. On the other hand, it is clear that, again, for $N$ sufficiently big, $\omega_i$ is $\e$-close to the concatenation $\gamma_r^i \cdot \omega_i$, reparametrized by $[i,i+1]$. Therefore, we see that the curve ${\gamma}_r^{[N]}$ is $2\e$-close, together with its derivatives, to the suitably parametrized concatenation
\begin{equation}\label{conc}
\gamma_r^0\cdot \omega_1\cdot \gamma_r^1 \cdot \omega_2 \cdot\ldots \cdot\gamma_r^{N-1} \cdot \omega_N.
\end{equation}
This means that for $\e$ small enough these curves are homotopic, since a linear interpolation between them would provide the required homotopy in $\LL M(\delta)$.

On the other hand, the curve (\ref{conc}) is homotopic to $\gamma_r\cdot \omega^N$ in $\LL M(\delta)$. Indeed, $\omega_i\cdot \gamma_r^i$ can be carried to $\gamma_r^i\cdot\omega_{i+1}$, with the help of the homotopy given by $$\gamma_r^i |_{[i/N,  i/N+\tau]}\cdot \exp_{\Gamma_r}(i/N+\tau,\alpha)\cdot \gamma_r^i |_{[ i/N+\tau, (i+1)/N]}.$$
Here $\tau\in [0,1]$ is the parameter of the homotopy. It remains to notice that $\gamma_r^0 \cdot\ldots \cdot\gamma_r^{N-1}=\gamma_r$ and $\omega_N=\omega$.
\end{proof}

\begin{remark}
In the course of the last proof we have, actually, shown that the left and the right multiplications by $\omega$ in $\LL M$ are homotopic.
\end{remark}

The curve 
$\gamma_r^{[N]}$ is homotopic
to $\gamma_r$. Hence, a representative of any homotopy class 
$\nu\in\pi_k \Omega M$ can be approximated in this way by a class in
the subspace consisting of non-degenerate curves.
Moreover, representing $\nu$ by curves which are constant in a neighbourhood of their endpoints, 
we can achieve that  the inital and the final Frenet frames of the corresponding 
non-degenerate curves coincide with the chosen frame. In particular, $\nu$ comes from a class is in $\pi_k \LL M$. 

\medskip

Now, with the principal bundle $$\SL(n)\to FTM\to M$$ we can associate
the fibration
$$\Omega \SL(n)\to \Omega FTM\to \Omega M.$$

Take some class $\eta\in \pi_k\Omega FTM$ and denote by $\nu$ its  projection to 
$\pi_k\Omega M$. Approximate $\nu$ by a class
$\nu'\in\pi_k \LL\M$ and let $F_*(\nu')$ be the image of $\nu'$ under the
Frenet frame map. The class $F_*(\nu')\eta^{-1}$ projects to zero in
$\pi_k\Omega M$, hence it is in the image of $\pi_k \Omega \SL(n)$.
This, in turn, can be approximated by a class in $\pi_k \LL\R^n$, well-defined up to
localization by $\alpha$, which, by Corollary~\ref{another}, gives rise to a class $c\in \pi_k\LL M$. Then $$F_*(c^{-1}\nu')=F_*(c)^{-1} F_*(\nu')\eta^{-1}\eta =  \eta.$$

Now, take a homotopy class  $c\in \pi_k\LL M$ that vanishes in $\pi_k FTM$. Approximating each curve in the null-homotopy of $c$ by non-degenerate curves using Lemma~\ref{appr} what we get is a null-homotopy of the image of $c$ in $\LL M_\omega$.  This shows that $\LL M_\omega$ and $\Omega FTM$ are weakly homotopy equivalent.

\section{On the group completion of $\LL M_+$}\label{theoneafter}

Here we shall show how Theorem~\ref{completion} follows from Theorem~\ref{main}. Let us first recall the group completion theorem \cite{MS, Qu}. Let $\M$ be a topological monoid and $\pi=\pi_0 \M$ the monoid of its path components. The homology of $\M$ is an algebra, with the product induced by that of $\M$, and $\pi$ can be thought of as a multiplicative subset of $H_*(\M)$. We say that the localization $H_*(\M)[\pi^{-1}]$ {\em can be calculated by right fractions} if the following conditions hold:
\begin{enumerate}
\item for every $s_1,s_2\in \pi$ there exist $t_1,t_2\in \pi$ with $s_1t_1=s_2t_2$;
\item given $s,s_1, s_2\in \pi$ such that $ss_1=ss_2$, there exists $t\in\pi$ with $s_1t=s_2t$;
\item given $r\in H_*(\M)$ and $s\in\pi$ with $sr=0$, there exists $t\in\pi$ with $rt=0$;
\item given $r\in H_*(\M)$ and $s\in\pi$ there exist $r'\in H_*(\M)$ and $t\in\pi$ with $rt=sr'$.
\end{enumerate}
The group completion theorem says that if the localization $H_*(\M)[\pi^{-1}]$ can be calculated by right fractions, then
$$
      H_*(\M)[\pi^{-1}]\approx H_*(\Omega B\M),
$$
where the homology is taken with arbitrary (possibly, twisted) coefficients. Moreover, the localization with respect to $\pi$ is induced by the canonical map $\M\to\Omega B\M$.

\begin{lemma}
The localization $H_*(\LL M_+)[\pi_0\LL M_+^{-1}]$ can be calculated by right fractions. 
\end{lemma}
\begin{proof} Let $\omega\in\pi=\pi_0\LL M$ be a class coming from $\pi_0\LL \R^n$. 
Observe that we have the following two facts: 
\begin{itemize}
\item[(a)] $\pi$ becomes a group if the class of $\omega$ is made invertible;
\item[(b)] the class of $\omega$ is in the centre of $\pi$; moreover, by the remark after Lemma~\ref{appr2}, the left and the right multiplications by $\omega$ give the same map on $H_*(\LL M_+)$.
\end{itemize}
For $t\in\pi$ denote by $[t]$ its class in the group completion of $\pi$. By (a) there exist $u_1$ and $u_2$ such that $u_1=[s_1]^{-1}$ and $u_2=[s_2]^{-1}$. Therefore,
$$s_1\cdot u_1\omega^{k}= s_2\cdot u_2\omega^{k}$$
for some $k$; this establishes (1). In order to prove (2) notice that if  $ss_1=ss_2$,  we have $[s_1]=[s_2]$. In particular, $s_1\cdot \omega^{k}= s_2\cdot \omega^{k}$
for some $k$.

Now, for $r\in H_* (\LL M_+)$ let $[r]$ be its image in $H_* (\Omega FT M)$. If $[r_1]=[r_2]$, there exists $k$ such that $r_1\omega^{k}=r_2\omega^{k}$. In order to establish (3), note that $sr=0$ implies $[s][r] =0$ and, hence,  $[r]=0$. This means that $r\cdot\omega^k=0$ for some $k$.
Finally, for $s\in\pi$ find $u$ with $[u]=[s]^{-1}$. Then $[s][u][r]=[r]$ in $H_* (\Omega FT M)$, which implies that $r\cdot\omega^{k}=s\cdot ur\omega^{k}$ for some $k$. This establishes (4).
\end{proof}

In particular, by the group completion theorem, $H_*(\Omega B\LL M_+)$ is obtained from $H_*(\LL M_+)$ by localization with respect to $\pi_0 \LL M_+$. 

On the other hand, the map $\LL M_+\to \Omega B \LL M_+$ factors through $\LL M_{\omega}$. Indeed, the multiplication by $\omega$ (or any other element) in  $\LL M_+$ is sent to a multiplication by an invertible, up to homotopy, element in $\Omega  B \LL M_+$; the localization with respect to an invertible element has no effect. Note that the homology of $\LL M_{\omega}$ is obtained from $H_*(\LL M_+)$ by localization with respect to the class of $\omega$. Localizing it any further with respect to other classes of $\pi_0\LL M_+$ has no effect since $\pi_0 \LL M_{\omega}$ is already a group. In particular, this means that the map $\LL M_\omega\to \Omega  B \LL M_+$ induces an isomorphism in homology.

The standard argument shows that the fundamental group of $\LL M_\omega$ is abelian. Indeed, since $\omega$ commutes, up to homotopy, with any compact subset of  $\LL M$,
the homotopy groups of the space $\LL M_\omega$ have the same properties as the homotopy groups of an $h$-space, and $h$-spaces have abelian fundamental groups. As a consequence, the homology equivalence  $\LL M_\omega\to \Omega  B \LL M_+$ is a weak homotopy equivalence. Now, applying Theorem~\ref{main}, we see that $\Omega  B \LL M_+$ is weakly homotopy equivalent to $\Omega FTM$. Since both spaces have homotopy types of CW-complexes (see \cite{Milnor}), they are homotopy equivalent.

\end{document}